\documentclass{article}

\usepackage{amsmath}
\usepackage{amssymb}
\usepackage{amsthm}
\usepackage{mathrsfs}
\usepackage{epsfig}
\usepackage{graphicx}
\usepackage[colorlinks]{hyperref}
\usepackage{authblk}

\newtheorem*{thm}{Theorem}
\newtheorem{theorem}{Theorem}[section]
\newtheorem{lemma}{Lemma}[section]

\newtheorem{corollary}{Corollary}[section]
\newtheorem{example}{Example}
\newtheorem{remark}{Remark}

\begin{document}
\title{Realization of Arbitrary Configuration of Limit Cycles of Piecewise Linear System}
\author[a,b]{Shaoqing Wang}
\author[a,c]{Jiazhong Yang}
\affil[a]{\small \it School of Mathematical Sciences, Peking University, Beijing 100871, China}
\affil[b]{wangshaoqing@pku.edu.cn}
\affil[c]{jyang@math.pku.edu.cn}
\date{}

 \maketitle
\begin{abstract}
In this paper, we consider the realization of configuration of limit cycles of piecewise linear systems on the plane. We show that any configuration of Jordan curves can be realized by a discontinuous piecewise linear system with two zones separated by a continuous curve.\\
\textbf{Keywords:} Piecewise linear system; Limit cycles; Configuration\\
{\bf 2000 Mathematics Subject Classification.} 34C05, 34C07, 34C25.
\end{abstract}

\section{Introduction}
One of the main branches of the qualitative theory of ordinary differential equations is to determine the number and configuration of limit cycles of planar polynomial systems, i.e. isolated periodic solutions. It can be traced back to Hilbert's sixteenth problem. As is known to all, there is a large amount of study on this problem and many good results have been obtained, but there is still a long distance to solve it completely.\\

Apart from the classical planar polynomial systems, planar piecewise smooth systems also absorb lots of attention recently, especially the bifurcation theory in piecewise smooth systems. For example, \cite{CGP} and \cite{LH} show us the application of  Hopf and homoclinic bifurcation techniques respectively in the estimation of the number of limit cycles (the definition will be given below) of planar piecewise smooth systems.\\

Among all types of the planar piecewise smooth systems, a special kind of system, i.e. piecewise linear system (PWL) is the simplest case to analyse. Except for its simpleness, people focus on the number and configuration of limit cycles of PWL for another reason that piecewise linear systems widely appear in the models of biology, chemistry and physics, which can be seen in\cite{ABB,BBCK}.\\

The number of limit cycles of piecewise linear system, by the bifurcation techniques and Melnikov methods, has drawn most attention, for instance see \cite{ALMT,BM,CLV,CLXY,CT,EL,FPRT,FPT,HY1,HY2,HY3,HZ,LC,LNT,LP,MT}. Among them, a very important case is the piecewise linear systems with two zones separated by a straight line, in which many people are interested, such as \cite{FPRT,FPT,HY1,HY3,HZ,LC,LP,MT}. Especially, in \cite{FPRT} and \cite{LC}, the continuous piecewise linear vector fields are deeply discussed by authors and they get a conclusion that there are at most one limit cycle in this kind of systems. On the other hand, in the work of \cite{FPT,HY1,LP}, the existence of a discontinuous PWL with two zones separated by a straight line having 3 limit cycles has been proved. But an open problem that whether 3 is the maximum number of limit cycles of discontinuous PWL with two zones separated by a straight line always haunts people.\\

Furthermore, if the separation line of PWL with two zones becomes a non-regular curve, more than three limit cycles can be obtained, see \cite{CT,LPZ}. Especially, in\cite{LPZ}, a piecewise linear system with two zones having $n$ limit cycles for any given integer $n\in \mathbf{N}$ has been obtained. Besides, the piecewise linear vector field with 3 zones also keeps the attention of many people, see the results of \cite{CLXY}.\\

Configuration of limit cycles in piecewise linear system, as important as the number of limit cycle, however, has been discussed much less than the number of limit cycles. Therefore, the aim of this paper is to consider the configuration of limit cycles of discontinuous piecewise linear systems with two zones. Different from the previous process in which we want to find a configuration of a given system, the reverse question that whether the piecewise linear system that realize a given configuration exists or not engages our interest, (later we will give its precise definition), similar to purpose of the work on polynomial system in \cite{LR}, finished by L. Jaume and G. Rodriguez.\\

As for the points in the separation of piecewise linear system, we will use the classical Filippov's terminology about sewing, sliding and escaping points, see \cite{F}. We give the definition as follows.\\

Denote by $Z=(X,Y)$, a discontinuous piecewise linear system with two zones with a boundary $\Gamma$, in one zone linear vector field $X$ is set and in the other we have the linear vector field $Y$. Then there are three types of points in the separation $\Sigma$: the \textit{sliding point} $p\in \Gamma$ means that $X(p)$ and $Y(p)$ point inward $\Sigma$; if $X(p)$ and $Y(p)$ point outward $\Sigma$, we call that $p\in \Gamma$ is the \textit{escaping point}; similarly, $p\in \Gamma$ is called the \textit{sewing piont} if $X(p)$ and $Y(p)$ point to the same direction and are transverse to $\Sigma$.\\

Recall that a limit cycle in a classical planar system is an isolated periodic orbit. Here for a piecewise linear system, a limit cycle, a special periodic solution, also has the property of isolation, as well as the condition that it intersects with the separation line at sewing points only.\\

By \textit{a configuration of limit cycles}, we mean a finite number of disjoint Jordan curves in $\mathbf{R}^2$. More exactly, it is a set $\mathcal{C}$ consisting of $C_1,\cdots,C_n$, which are disjoint simple closed curves, such that $C_i\cap C_j=\emptyset$ ($i\neq j$). And two configurations $\mathcal{C}=\{C_1,\cdots,C_n\}$ and $\mathcal{C^{'}}=\{C^{'}_1,\cdots,C^{'}_m\}$ are called \textit{topologically equivalent} if there exists a function $H:\mathbf{R}^2\rightarrow\mathbf{R}^2$, such that $H$ is a homeomorphism and $H(\cup^n_{i=1}C_i)=\cup^m_{j=1}C^{'}_{j}$. Trivially, if $\mathcal{C}$ and $\mathcal{C^{'}}$ are topologically equivalent, then $n=m$. In the following part of this paper, we only consider the equivalence class of configuration.\\

Another important concept in this paper must be given here. If the configuration of a piecewise linear system $(X)$ is topologically equivalent to a given configuration $\mathcal{C}$, we say that $(X)$ \textit{realize} the configuration $\mathcal{C}$.\\

With these notation, the main result of this paper about the realization of a given configuration can be stated as follows:
\begin{thm}[A]
Let $\mathcal{C}$ be any given configuration of limit cycles, there exists a piecewise linear system with two zones to realize it.
\end{thm}

We shall prove this theorem and show the construction process of the piecewise linear system in the following sections. In section 2, we give some lemmas and we use these lemmas to prove the main theorem in section 3.\\

\section{Some Lemmas}
Before we pay attention to construction of piecewise linear system with two zones to realize any given configuration of limit cycles, some preparation are needed.\\

Suppose $f(x)$ is a well-defined continuous function on $\textbf{R}$ and $f(0)\geq 0.$ Then its graph $\{(x,f(x))|x\in\textbf{R}\}$ separates the whole plane $\textbf{R}^2$ into two zones. In the lower part $\{(x,y)|y<f(x),x\in \mathbf{R}\}$, we set a linear system $(X_1)$ of which the singularity is a virtual center at $(0,a)$, where $a>f(0)$; i.e. $(X_1): \dot{x}=-(y-a), \dot{y}=x$. In the upper domain $\{(x,y)|y>f(x),x\in \mathbf{R}\}$, we also have a linear system $(X_2): \dot{x}=-y, \dot{y}=x$ with virtual center at $(0,0)$. Obviously, the first integrals of $(X_1)$ and $(X_2)$ are $H_1(x,y)=x^2+(y-a)^2$ and $H_2(x,y)=x^2+y^2$, respectively. Hence a piecewise linear system $X=(X_1,X_2)$ exists.\\

In the following lemma, we give a necessary condition of existence of periodic orbits in the system $X$.
\begin{lemma}\label{sepa1}
If there exists a periodic orbit $C$ in the above piecewise linear system $X$ and $C$ intersects the separation line at $(x_1,f(x_1))$ and $(x_2,f(x_2))$, where $x_1\neq x_2$, then $x_1=-x_2$ and $f(x_1)=f(x_2)$, it means that $(x_1,f(x_1))$ and $(x_2,f(x_2))$ are symmetric with respect to $y$-axis.
\end{lemma}
\begin{proof}
From the condition, we have that
\begin{eqnarray*}
&&(x_1,f(x_1))\ \text{and} \ (x_2,f(x_2)) \ \text{are in}\ C,\\
&\Leftrightarrow& (x_1,f(x_1)) \ \text{and} \  (x_2,f(x_2))\ \text{are in the same orbits of}\ (X_1)\ \text{and} \ (X_2),
\end{eqnarray*}
which is equivalent to $H_1(x_1,f(x_1))=H_1(x_2,f(x_2))$ and $H_2(x_1,f(x_1))=H_2(x_2,f(x_2))$, namely,
\begin{eqnarray*}
\left\{
\begin{aligned}
&x_1^2+(f(x_1)-a)^2=x_2+(f(x_2)-a)^2,&\\
&x_1^2+(f(x_1))^2=x_2^2+(f(x_2))^2,&
\end{aligned}
\right.
\end{eqnarray*}
that is
\begin{eqnarray*}
\left\{
\begin{aligned}
&x_1=-x_2,&\\
&f(x_1)=f(x_2).&
\end{aligned}
\right.
\end{eqnarray*}
Now we complete the proof of this lemma.
\end{proof}

\begin{example}\label{ex1}
Based on Lemma 2.1, a special continuous function $f(x)$ can be chosen, such that the periodic orbit which intersects with the graph of $f(x)$ is isolated, in other word, it is a limit cycle.\\

we can construct a function as follows:
\begin{equation*}
f(x)=
\left\{
\begin{aligned}
\frac{1}{2} &,& x>\frac{3}{2}\\
\frac{1}{2}\sin(\pi x)+1&,&-\frac{3}{2}\leq x\leq\frac{3}{2}\\
\frac{3}{2}&,& x<-\frac{3}{2}
\end{aligned}
\right.
\end{equation*}
and we choose $a=2$. By computing the slopes, a periodic orbit appears and passes through $(1,1)$ and $(-1,1)$. Since there are no other pairs of points symmetric with respect to $y-$axis, there are no other periodic solutions based on the conclusion of Lemma \ref{sepa1}. In the following part of this paper, we will say that the only periodic solution (i.e. limit cycle) passes through $(1,1)$ since $(1,1)$ and $(-1,1)$ appear simultaneously in this limit cycle.
\end{example}

We also need a more complicated lemma. Before this lemma, we make some explanation of notations. Suppose $f(x)$ is a continuous function on $\mathbf{R}$ and $f(x)>0$, $\forall x \in \mathbf{R}$. Then its graph $\{(x,f(x))|x\in \mathbf{R}\}$ and $x$-axis as two separation lines divide whole plane $\mathbf{R}^2$ into three domains: the first domain $I_1$ is $\{(x,y)|y>f(x),x\in \mathbf{R}\}$, the second part $I_2$ is $\{(x,y)|0<y<f(x),x\in \mathbf{R}\}$, and the third one $I_3$ defined by $\{(x,y)|y<0,x\in \mathbf{R}\}$. Now we construct a piecewise linear system. In the first and third domains $I_1$ and $I_3$, the vector field $Y_1$,$Y_3$ are dominated by the system $\dot{x}=-y$, $\dot{y}=x$. In the second domain $I_2$, the linear vector $Y_2$ is $\dot{x}=-(y-a)$, $\dot{y}=x$, where $a\geq f(0)$. Then we have a piecewise linear system $Y=(Y_1,Y_2,Y_3)$. A necessary condition of periodic orbit of system $Y$ is given in the following lemma.
\begin{lemma}\label{sepa2}
Suppose there exists a periodic orbit in $Y$ which intersects with the graph of $f(x)$. Then that two points $(x_1,f(x_1))$ and $(x_2,f(x_2))$ are at this periodic orbit leads to the following condition:
\begin{equation*}
\left\{
\begin{aligned}
&x_1=-x_2,&\\
&f(x_1)=f(x_2).&
\end{aligned}
\right.
\end{equation*}
\end{lemma}

\begin{proof}
On account of the definition of periodic orbit of piecewise linear system, if a periodic orbit $C$ intersects with the separation lines of $Y$, it must intersect with the graph of $f(x)$ because of the vector field of $Y$.\\

Since there is another separation line, namely, $x$-axis, there are two cases to be considered: $C$ intersects $x$-axis in the first case, the other case is that $C$ does not intersects $x$-axis.\\

The second case is the same as that in Lemma 2.1, so now we pay attention to the first case. Suppose that periodic orbit $C$ intersects $x$-axis at two points $(x_3,0)$ and $(x_4,0)$ as well, where $x_3$ and $x_4$ may be the same number. The figure 1 shows the meaning. ($A$, $B$ represent $(x_1,f(x_1))$ and $(x_2,f(x_2))$ respectively, and $C$, $D$ are $(x_3,0)$ and $(x_4,0)$ respectively.)\\
\includegraphics[width=0.4\textwidth,angle=0]{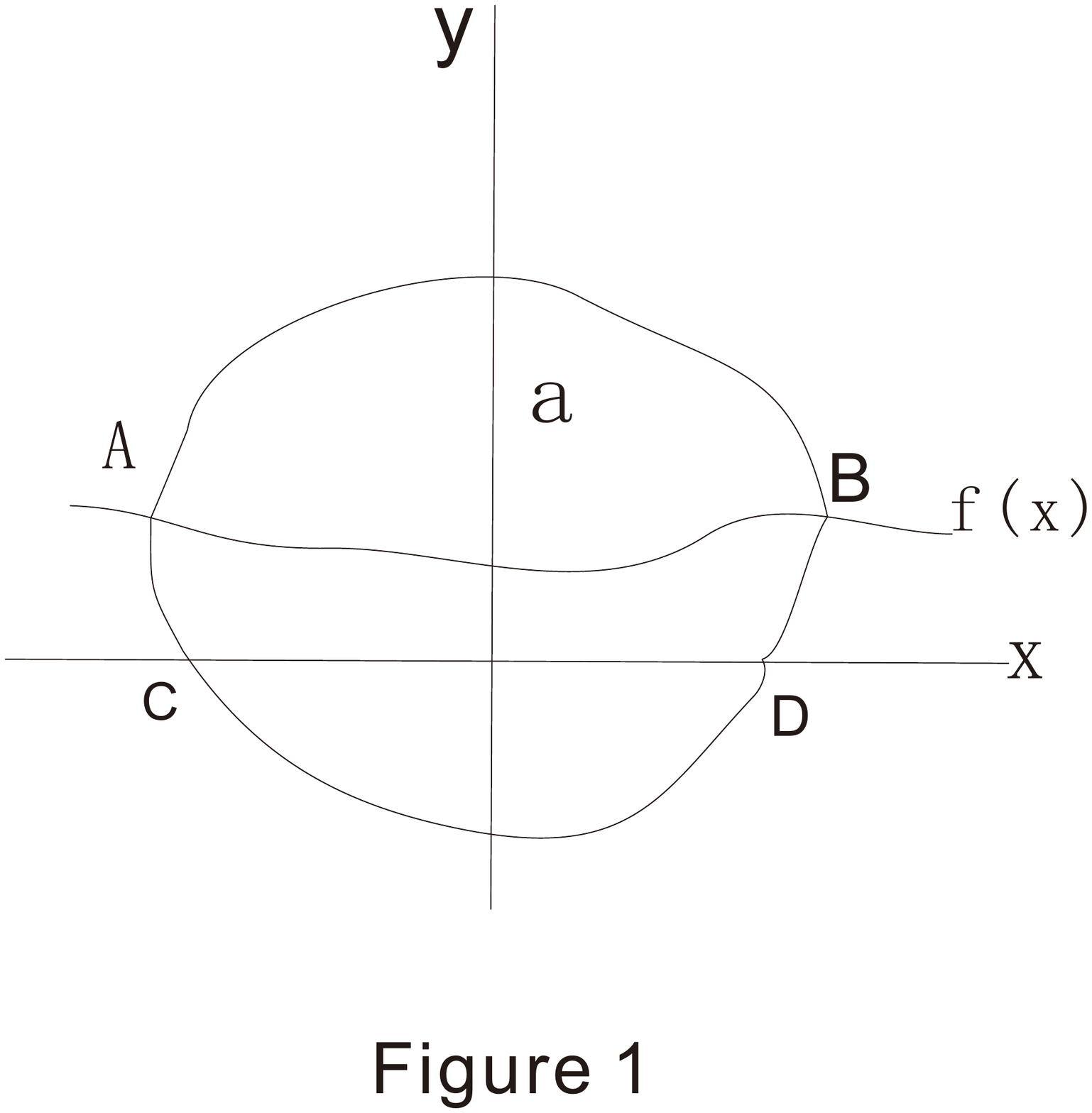}

Based on these notations, we have the following equivalent condition:
\begin{eqnarray*}
&&\left\{
\begin{aligned}
&x_3^2=x_4^2,&\\
&x_1^2+(f(x_1))^2=x_2^2+(f(x_2))^2,&\\
&x_1^2+(f(x_1-a))^2=x_2^2+(f(x_2)-a)^2.&
\end{aligned}
\right.\\
&\Leftrightarrow&
\left\{
\begin{aligned}
&x_3^2=x_4^2,&\\
&x_1^2=x_2^2,&\\
&f(x_1)=f(x_2).&
\end{aligned}
\right.
\end{eqnarray*}
Combining with the definition of periodic orbit of piecewise linear system, above condition is equivalent to
\begin{equation*}
\left\{
\begin{aligned}
&x_3=-x_4,&\\
&x_1=-x_2,&\\
&f(x_1)=f(x_2).&
\end{aligned}
\right.
\end{equation*}

From these two cases, we know that whether periodic orbit $C$ intersects $x$-axis or not, the conclusion of this lemma holds.
\end{proof}

\begin{example}
Similar to the previous Example \ref{ex1} below Lemma \ref{sepa1}, we can choose the same function $f(x)$, namely,
\begin{equation*}
f(x)=
\left\{
\begin{aligned}
\frac{1}{2} &,& x>\frac{3}{2}\\
\frac{1}{2}\sin(\pi x)+1&,&-\frac{3}{2}\leq x\leq\frac{3}{2}\\
\frac{3}{2}&,& x<-\frac{3}{2}
\end{aligned}
\right.
\end{equation*}
and $a=2$. Then there exists a periodic orbit $C$ which pass through $(1,1)$ and does not intersect with $x-$axis. Besides, from the conclusion of Lemma \ref{sepa2}, there are no other periodic orbits as there are no other pairs of symmetric points, then $C$ becomes the only limit cycle in the piecewise linear system.
\end{example}

\begin{remark}
The function $f(x)$ in the previous lemmas need not be defined on $\mathbf{R}$, they can be defined on a subset of $\mathbf{R}$.
\end{remark}

To prepare for the proof of the Theorem(A), we need another lemma. Also we give some terminologies firstly. Assume that $f(x)>0$ is a continuous function defined on $x\leq7$ and $f(7)=0$, then its graph $\{(x,f(x))|x\leq 7\}$ and the straight line $\{(x,0)|x\leq 7\}$ separate whole plane into two parts: the smaller part $I_1=\{(x,y)|0<y<f(x),x<7\}$ and the bigger part $I_2=\mathbf{R^2}\setminus I_1$. In the part $I_1$, the linear vector $Z_1$ is $\dot{x}=-(y-a)$, $\dot{y}=x$, where $a\geq f(0)>0$. In the part $I_2$, we set a linear vector $Z_2$ is $\dot{x}=-y$, $\dot{y}=x$. Consequently, we have a piecewise linear system $Z=(Z_1,Z_2)$. Now we will show the third lemma based on the above notation.
\begin{lemma}\label{thirdlemma}
In the above piecewise linear system $Z=(Z_1,Z_2)$, there are no periodic solutions which intersect the straight line $\{(x,0)|x\geq7\}$.
\end{lemma}
\begin{proof}
We shall prove this lemma by contradiction. Suppose that there exists a periodic solution $C$ that passes through $(x_1,0)$ where $x_1\geq7$, then it must pass through $(-x_1,0)$ by symmetry of linear system $Z_2$. Also assume that this periodic solution $C$ intersects the graph of $f(x)$ at $(x_2,f(x_2))$.\\

Then we have the following formulas
\begin{eqnarray*}
&&\left\{
\begin{aligned}
&x_1^2=x_2^2+(f(x_2))^2,&\\
&x_1^2+a^2=x_2^2+(f(x_2)-a)^2.&
\end{aligned}
\right.
\end{eqnarray*}
which induce $a=0$. This is a contradiction. Then we complete the proof.
\end{proof}

\section{Proof of Theorem (A)}
Now, based on these preparation, we can construct a piecewise linear system to realize a special configuration $\mathcal{C}=\{C_0, C_1,\cdots, C_n\}$, where $C_0$ is a limit cycle containing $C_1$, $C_2$, $\cdots$, $C_n$, and $C_1$, $C_2$, $\cdots$, $C_n$ are all limit cycles which satisfy a condition that each of them does not contain one another. For convenience, we will say that $C_1$, $C_2$, $\cdots$, $C_n$ satisfy coordination condition or $C_1$, $C_2$, $\cdots$, $C_n$ are coordinated.\\

We put the method of construction into the following theorem.
\begin{theorem}\label{coordi1}
For a configuration $\mathcal{C}=\{C_0, C_1,\cdots,C_n\}$ of limit cycles in which $C_0$ contains $C_1$, $C_2$, $\cdots$, $C_n$ and $C_1$, $C_2$, $\cdots$, $C_n$ are coordinated ,there exists a piecewise linear system which realizes it.
\end{theorem}
\begin{proof}
Without loss of generality, we assume $n=3$ since the method of construction is similar.\\

For the configuration $\mathcal{C}=\{C_1, C_2,\cdots, C_n\}$ in the condition of the theorem, we shall use Figure 2 to construct the piecewise linear system.\\

\includegraphics[width=0.53\textwidth,angle=0]{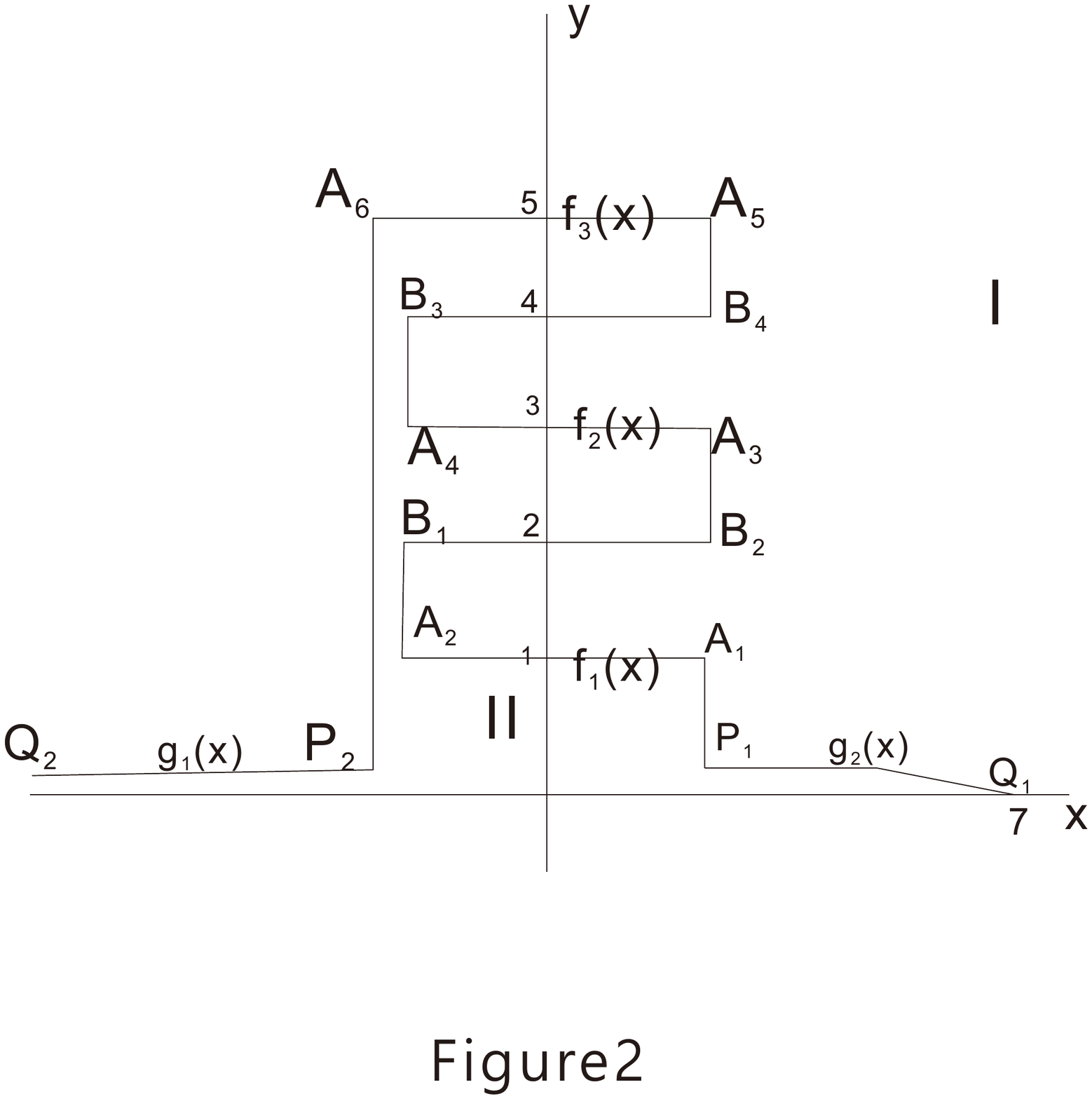}

In this figure, $Q_1$ is a fixed point at $(7,0)$; $P_1A_1$, $A_2B_1$, $B_2A_3$, $A_4B_3$, $B_4A_5$ and $A_6P_2$ are vertical; $B_1B_2$, $B_3B_4$ are straight lines perpendicular to $y-$axis; and $A_2A_1$, $A_4A_3$, $A_6A_5$, $Q_2P_2$ and $P_1Q_1$ are the graphs of functions $f_1(x)$, $f_2(x)$, $f_3(x)$, $g_1(x)$ and $g_2(x)$ respectively. Here we choose these functions as follows:
\begin{eqnarray*}
f_1(x)&=&\frac{1}{4}\sin(2\pi x)+1, x\in [-\frac{3}{4},1],\\
f_2(x)&=&\frac{1}{4}\sin(2\pi x)+3, x\in [-\frac{3}{4},1],\\
f_3(x)&=&\frac{1}{4}\sin(\frac{3}{2}\pi x)+5, x\in [-1,1],\\
g_1(x)&=&\frac{1}{32}\sin\pi(x+\frac{13}{2})+\frac{1}{16}, x\leq -1,
\end{eqnarray*}
and
\begin{eqnarray*}
g_2(x)&=&
\left\{
\begin{aligned}
&\frac{1}{8},& 1\leq x \leq 6,\\
&-\frac{1}{8}(x-7),& 6\leq x \leq 7.
\end{aligned}
\right.
\end{eqnarray*}
Now the continuous curve $Q_2P_2A_6A_5B_4B_3A_4A_3B_2B_1A_2A_1P_1Q_1$ and the subset of $x-$axis $\{x,0|x\leq 7\}$ together form the separation line, which divides the whole plane $\mathbf{R}^2$ into two connected domains $I, II$. In the domain $I$, we set the system $Y_1:\dot{x}=-y$, $\dot{y}=x$; In the domain $II$, we set $Y_2:\dot{x}=-(y-5)$, $\dot{y}=x$, then the piecewise linear system $Y=(Y_1,Y_2)$ exists. \\

From the expression of $f_1(x)$ we know that only $(\frac{1}{2},1)$ and $(-\frac{1}{2},1)$ are symmetric with respect to $y-$axis in the graph of $f_1(x)$. And the expression of $f_1(x)$ guarantees the existence a periodic orbit $C_1$ which passes through $(\frac{1}{2},1)$. So this periodic orbit $C_1$ must be a limit cycle.\\

Similarly, by direct computation and the condition of Lemma \ref{sepa1} and \ref{sepa2}, there exist only two more limit cycles $C_2$, $C_3$ which intersect graphs of $f_2(x)$ and $f_3(x)$ at $(\frac{1}{2},3)$ and $(\frac{2}{3},5)$ respectively. Besides, there is another limit cycle $C_0$ which passes through the graph of functions $g_1(x)$ and $g_2(x)$ at $(-\frac{13}{2},\frac{1}{16})$ and $(\frac{13}{2},\frac{1}{16})$ respectively. And $(-\frac{13}{2},\frac{1}{16})$ and $(\frac{13}{2},\frac{1}{16})$ are only pair of points which are symmetric with respect to $y-$axis and in the graphs of $g_1(x)$ and $g_2(x)$, so there are no other limit cycles which intersect graphs of $g_1(x)$ and $g_2(x)$. By similar method to that in Lemma \ref{thirdlemma}, there are also no other limit cycles which intersect the graph of $g_1(x)$ and the straight line $\{(x,0)|x>7\}$.\\

In summary, there are only four limit cycles $C_0$, $C_1$, $C_2$ and $C_3$ in this piecewise linear system $Y$. $C_0$ passes through $(\frac{13}{2},\frac{1}{16})$, and $C_1$, $C_2$ and $C_3$ pass through $(\frac{1}{2},1)$, $(\frac{1}{2},3)$ and $(\frac{2}{3},5)$, respectively. Thus $C_0$ contains $C_1$, $C_2$ and $C_3$, meanwhile $C_1$, $C_2$, $C_3$ are coordinated. Then this theorem has been proved.
\end{proof}

The configuration in Theorem \ref{coordi1} is very special, there exists more complicated configuration $\mathcal{C}=\{C_{01},\cdots, C_{0k_0},C_{11}, \cdots, C_{1k_1},\cdots, C_{n1}, \cdots, C_{nk_n}\}$, more exactly, $C_{0j}$ is contained in $C_{0(j-1)}$, $(j=2,\cdots, k_0)$, $C_{0k_0}$ contains $C_{11}$, $C_{21}$, $\cdots$, $C_{n1}$ and $C_{11}$,$C_{21}$, $\cdots$, $C_{n1}$ are coordinated, and $C_{ij}$ is contained in $C_{i(j-1)}$, $(i=1,\cdots, n,j=2,\cdots, k_n)$. We shall construct a piecewise linear system with two zones to realize this kind of configuration. For convenience, this kind of configuration is called deep coordinated configuration of limit cycles.

\begin{corollary}\label{corollary}
For any given deep coordinated configuration of limit cycles, there exists a piecewise linear system to realize it.
\end{corollary}

\begin{proof}
Without loss of generality, we can assume that $n=3$, $k_0=1$, $k_1=k_2=2$, $k_3=1$. We also use the separation line in Theorem \ref{coordi1} shown in Figure 2.\\

The work remained to us is to choose special functions $f_1(x)$ and $f_2(x)$. $f_3(x), g_1(x)$ and $g_2(x)$ are the same as that in Theorem \ref{coordi1}.\\

Let
\begin{eqnarray*}
f_1(x)=\frac{1}{4}\sin (3\pi x)+1, x\in[-\frac{3}{4},1]\\
f_2(x)=\frac{1}{4}\sin (3\pi x)+3, x\in[-\frac{3}{4},1]
\end{eqnarray*}
Then there are only two pairs of points in graph of $f_1(x)$: ${(\frac{1}{3},1),(-\frac{1}{3},1)}$ and ${(\frac{2}{3},1),(-\frac{2}{3},1)}$ which are symmetry with respect to $y-$axis. And the limit cycles $C_{11}$($C_{12}$, resp) which passes through $(\frac{1}{3},1)$ ($(\frac{2}{3},1)$, resp) exists. So based on the conclusion of Lemma \ref{sepa1}, there are no other limit cycles which intersect the graph of $f_1(x)$.\\

Similarly, there exist only two limit cycles $C_{21}$ and $C_{22}$ which intersect the graph of $f_2(x)$ at $(\frac{1}{3},3)$, $(\frac{2}{3},3)$ respectively.\\

Recall the conclusion in the proof of Theorem \ref{coordi1}, there are also limit cycles $C_0$, $C_3$ and these limit cycles satisfy that $C_0$ contains $C_{11}$, $C_{12}$, $C_{21}$, $C_{22}$, $C_{3}$ and $C_{11}$, $C_{21}$, $C_{3}$ are coordinated, $C_{i2}$ is contained in ${C_{i1}}$ $(i=1,2)$. Thus we construct a piecewise linear system to realize the deep coordinated configuration for $n=3, k_0=1, k_1=k_2=2, k_3=1$.\\

Using the same method, we can obtain a piecewise linear system to realize any given deep coordinated configuration of finite limit cycles.
\end{proof}

Now, we can only realize deep coordinated configuration of limit cycles. More energy should be spent to prove Theorem (A). From now on, we shall focus on the construction of piecewise linear system to realize more general given configuration of limit cycles.

\begin{theorem}\label{thm2}
For a configuration $\mathcal{C}=\{C_0,C_1,C_2,C_3,C_{11},C_{12}\}$ in which $C_{0}$ contains coordinated $C_1$, $C_2$, $C_3$, and $C_{11}$, $C_{12}$ are coordinated limit cycles that are both contained in $C_{1}$. (see Figure 3). Then there exists a piecewise linear system to realize it.
\end{theorem}

\includegraphics[width=0.4\textwidth,angle=0]{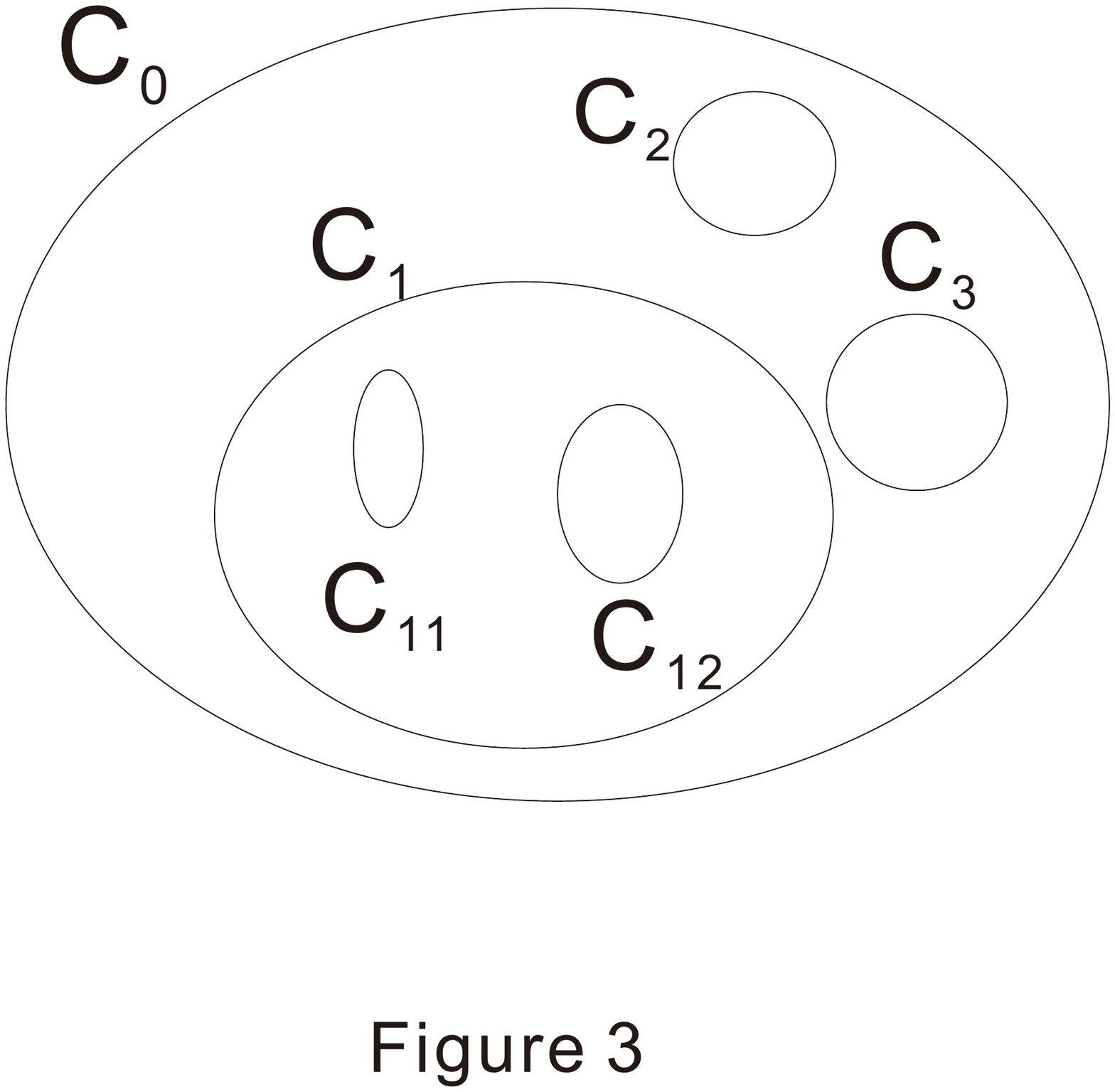}
\includegraphics[width=0.6\textwidth,angle=0]{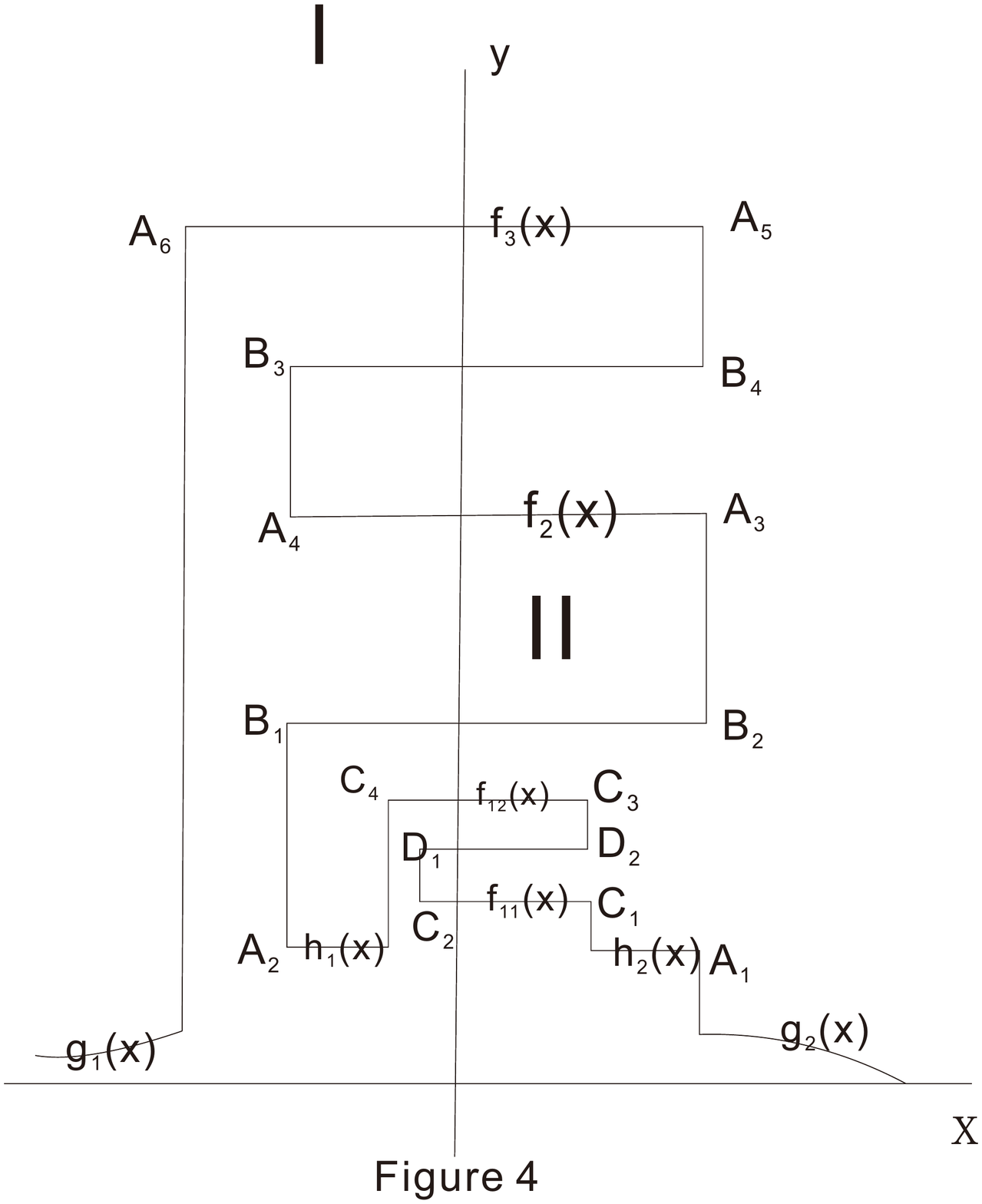}
\begin{proof}
To realize this configuration, we will modify the curve in Theorem \ref{coordi1} to get a new one.\\

Approximately, our separation line resembles the curve in the above figure. (See Figure 4)\\

In this figure, all vertical lines are straight, so are three horizontal lines $B_1B_2$, $B_3B_4$, $D_1D_2$ which pass through $(0,2)$, $(0,4)$ and $(0,\frac{98}{96})$ respectively. And other lines are the graphs of corresponding functions. Here $g_1(x)$, $g_2(x)$, $f_2(x)$, $f_3(x)$ are the same as those in Theorem \ref{coordi1}, namely,
\begin{equation*}
f_2(x)=\frac{1}{4}\sin(2\pi x)+3, x\in [-\frac{3}{4},1],
\end{equation*}
\begin{equation*}
f_3(x)=\frac{1}{4}\sin(\frac{3}{2}\pi x)+5, x\in [-1,1],
\end{equation*}
\begin{equation*}
g_1(x)=\frac{1}{32}\sin\pi(x+\frac{13}{2})+\frac{1}{16}, x\leq -1,
\end{equation*}
and
\begin{equation*}
g_2(x)=
\left\{
\begin{aligned}
&\frac{1}{8},& 1\leq x \leq 6,\\
&-\frac{1}{8}(x-7),& 6\leq x \leq 7.
\end{aligned}
\right.
\end{equation*}

Now we choose $f_{11}(x)$, $f_{12}(x)$ and $h_1(x)$, $h_2(x)$ as follows:
\begin{eqnarray*}
f_{11}=\frac{1}{192}\sin (32\pi x)+\frac{97}{96}&,& x\in [-\frac{1}{16},\frac{1}{16}]\\
f_{12}=\frac{1}{192}\sin (32\pi x)+\frac{99}{96}&,& x\in [-\frac{3}{64},\frac{1}{16}]\\
h_1(x)=\frac{1}{4}\sin(2\pi x)+1&,& x\in [-\frac{3}{4},-\frac{1}{16}]\\
h_2(x)=\frac{1}{4}\sin(2\pi x)+1&,& x\in [\frac{1}{16},1]
\end{eqnarray*}
Then this curve and the straight line $\{x,0|x\leq 7\}$ as a separation line divide $\mathbf{R}^2$ into two parts: the outer part $I$ and the inner part $II$.\\

In part $I$, we set $\dot{x}=-y,\dot{y}=x$, and in domain $II$, we have $\dot{x}=-(y-5),\dot{y}=x$. Thus we have a piecewise linear system. Then we shall show that this system can realize the configuration $\mathcal{C}=\{C_0,C_1,C_2,C_3,C_{11},C_{12}\}$ in the condition.\\

From the expression of $f_{11}(x)$, we know that only $(\frac{1}{32},\frac{97}{96})$ and $(\frac{1}{32},\frac{97}{96})$ at graph of $f_{11}(x)$ are symmetric with to $y-$axis. Similarly, only $(\frac{1}{32},\frac{99}{96})$ and $(\frac{1}{32},\frac{99}{96})$ at graph of $f_{12}(x)$ are symmetric with to $y-$axis. By direct computation, we know that there exist two limit cycles which pass through $(\frac{1}{32},\frac{97}{96})$ and $(\frac{1}{32},\frac{99}{96})$ respectively. And we call them limit cycles $C_{11}$ and $C_{12}$ respectively. Except for $C_{11}$ and $C_{12}$, there are no other limit cycles which intersect the graphs of $f_{11}(x)$ or $f_{12}(x)$, because in these two graph there are no other pairs of points which are symmetric with respect to $y-$axis.\\

Note that the expressions of $h_1(x)$ and $h_2(x)$ are the same as $f_1(x)$ in Theorem \ref{coordi1}, then there is only a limit cycle $C_1$ which intersects the graphs of $h_1(x)$ and $h_2(x)$. Combining the limit cycles $C_2$, $C_3$ and $C_0$ which intersect the graphs of $h_1(x)$, $f_2(x)$, $f_3(x)$ and $g_1(x)$, $g_2(x)$ respectively, as in Theorem \ref{coordi1}, we know that there only 6 limit cycles in this piecewise linear system, and their configuration is topologically equivalent to the configuration $\mathcal{C}=\{C_0,C_1,C_2,C_3,C_{11},C_{12}\}$ in the condition of Theorem \ref{thm2}. Thus we complete the proof.
\end{proof}

From the proof of Theorem \ref{thm2}, we can see that the curve in proof of Theorem \ref{coordi1} takes a crucial role in the construction of piecewise linear system. Since this kind of curve resembles a key to a door, we give it a name: \textit{Key Curve} with n teeth. The curve in Figure 2 is a Key Curve with 3 teeth.\\

To speak concisely, we say that the curve in proof of Theorem \ref{thm2} is obtained by substituting the first tooth in the Key Curve in Figure 2 for a smaller Key Curve with 2 teeth. So if more teeth in curve on the proof of Theorem \ref{coordi1} are replaced by smaller Key Curves, we can get more complicated separation line to realize more complicated configuration. This deduces the proof of Theorem (A) in section 1.\\

\begin{proof}[Proof of Theorem (A)]
Based on similar substitution method in the proof of Theorem \ref{thm2}, we can modify the separation line in Theorem \ref{coordi1} to construct a piecewise linear system $(S_1)$ to realize a given configuration $\mathcal{C}$ $=$ $\{$ $C_0$,$C_1$,$C_{11}$,$\cdots$, $C_{1k_1}$, $C_2$,$C_{21}$,$\cdots$, $C_{2k_2}$,$\cdots$,$C_n$,$C_{n1}$,$\cdots$, $C_{nk_n}\}$ where $C_0$ contains coordinated $C_1$, $C_2$, $\cdots$, $C_n$ and meanwhile $C_i$ contains coordinated $C_{i1}$, $C_{i2}$, $\cdots$, $C_{ik_i}$, $(i=1$, $2\cdots$ $n)$.\\

More generally, if there is a configuration $\mathcal{C}$ $=$ $\{C_0$, $C_1$,$C_1^{(2)}$,$\cdots,C_1^{(l)}$, $C_{11}$, $\cdots$, $C_{1k_1}$,$C_2$,$C_{21}$,$\cdots$, $C_{2k_2}$,$\cdots$,$C_n$,$C_{n1}$,$\cdots$, $C_{nk_n}\}$ where $C_1$ contains $C_1^{(2)}$, $\cdots$, $C_1^{(l)}$ and $C_1^{i}$ contains $C_1^{i+1}$, $(i=2,\cdots l-1)$, meanwhile $C_1^{l}$ contains $C_{11}$,$\cdots C_{1k_1}$, then we can also use the same method as that in Corollary \ref{corollary} to get a piecewise linear system $(S_2)$ to realize it.\\

If there is a more complicated configuration in which $C_{11}$ also contains coordinated $C_{111}$, $C_{112}$ $\cdots$ $C_{11m_{11}}$, then let us modify $(S_2)$ by replacing the tooth in the position of $C_{11}$ with a smaller Key Curve with $m_{11}$ teeth like the method in proof of Theorem \ref{thm2}, we obtain a new separation line, as well as a new piecewise linear system to realize this configuration.\\

If the limit cycle $C_0$ containing other limit cycles does not exist, we can choose special functions $g_1(x)$ and $g_2(x)$ to make it vanish.\\

For any given finite configuration of limit cycles, since the number of limit cycles is finite, we can continue the above steps finite times to obtain a piecewise linear system to realize it. Then we complete the proof of the Theorem (A).
\end{proof}

\end{document}